\documentclass[a4paper,12pt,twoside]{amsart}

\usepackage{fullpage}

\usepackage{mymacros,oitp}
\usepackage{hyperref}

\title[]{Reconstruction of general elliptic K3 surfaces from their Gromov--Hausdorff limits}
\author{Kenji Hashimoto}
\address{
Graduate School of Mathematical Sciences,
The University of Tokyo,
3-8-1 Komaba,
Meguro-ku,
Tokyo,
153-8914,
Japan.}
\email{hashi@ms.u-tokyo.ac.jp}
\author[K.~Ueda]{Kazushi Ueda}
\address{
Graduate School of Mathematical Sciences,
The University of Tokyo,
3-8-1 Komaba,
Meguro-ku,
Tokyo,
153-8914,
Japan.}
\email{kazushi@ms.u-tokyo.ac.jp}
\date{}
\begin{document}

\maketitle

\begin{abstract}
We show that a general elliptic K3 surface with a section
is determined uniquely by its discriminant,
which is a configuration of 24 points on the projective line.
It follows that a general elliptic K3 surface with a section
can be reconstructed from its Gromov--Hausdorff limit
as the volume of the fiber goes to zero.
\end{abstract}

\section{Introduction}
 \label{sc:introduction}

Let $f, g \in \bC[z,w]$ be homogeneous polynomials
of degree 8 and 12 respectively,
and define a polynomial $h$ of degree 24 by
\begin{align} \label{eq:def_h}
 h = f^3 + g^2.
\end{align}
We show the following in this paper:

\begin{theorem} \label{th:main}
If $f$ and $g$ are general,
then the decomposition of $h$
into the sum of a cube and a square is unique,
up to the obvious ambiguity
of multiplication by a cubic and a square root of unity. 
\end{theorem}

Note that the corresponding problem in number theory
has non-unique solutions in general.
There are three solutions
\begin{align}
 1^3 + 4^2 &= 2^3 + 3^2, &
 1^3 + 8^2 &= 4^3 + 1^2, &
 2^3 + 9^2 &= 4^3 + 5^2
\end{align}
for positive integers less than 100,
and
\begin{align}
 1^3 + 32^2
  = 4^3 + 31^2
  = 5^3 + 30^2
  = 10^3 + 5^2
\end{align}
is the smallest integer
which can be written as the sum of a cube and a square
in more than two ways.
An elliptic curve of the form
\begin{align}
 y^2 = x^3 + n
\end{align}
for a non-zero integer $n$
is known as a \emph{Mordell curve},
and has been studied for many years.

The decomposition over the function field is not unique
if $f$ and $g$ are not general.
For example,
for homogeneous polynomials $u$ and $v$ of degree 4,
one has
\begin{align}
 4 ( u v )^3 + (u v (u-v))^2
  &= 0^3 + (u v (u+v))^2.
\end{align}

An elliptic K3 surface with a section
has a Weierstrass model
\begin{align}
 y^2 = 4 x^3 - g_8(z,w) x - g_{12}(z,w)
\end{align}
in $\bP(4,6,1,1)$,
where $g_8(z,w)$ and $g_{12}(z,w)$ are homogeneous polynomials
in $z$ and $w$ of degree 8 and 12 respectively.
The discriminant is given by
\begin{align}
 \Delta = g_8^3 - 27 g_{12}^2,
\end{align}
which is a homogeneous polynomial of degree 24.
The decomposition problem asks
if $g_8$ and $g_{12}$ can be reconstructed
from $\Delta$.
Our interest in this problem comes from mirror symmetry.
It is shown in \cite{MR1863732,MR3514555}
that a sequence of K\"ahler--Einstein metrics
on an elliptic K3 surface
with a fixed diameter
converges
in the Gromov--Hausdorff topology
to a sphere with a Monge--Amp\`ere structure with singularities,
if the volume of the fiber goes to zero.
Since the complement of the discriminant of the elliptic K3 surface
as a punctured Riemann sphere
can be reconstructed as the conformal class
of the smooth part of the limit metric,
\pref{th:main} allows the reconstruction of
(the Jacobian of) a general elliptic K3 surface
from the Gromov--Hausdorff limit.

This paper is organized as follows:
In \pref{sc:surfaces},
we associate an auxiliary elliptic surface $X$ of general type
with a polynomial $h = f^3 + g^2\in \bC[z,w]$ of degree 24.
In \pref{sc:rho},
we show that the Picard number of $X$ is 4
if $f$ and $g$ are very general.
In \pref{sc:proof},
we give a proof of \pref{th:main}.
In \pref{sc:mirror},
we discuss the relationship with mirror symmetry for K3 surfaces.

\subsection*{Acknowledgements}
We thank Yuichi Nohara for an unpublished collaboration
on Gromov--Hausdorff compactifications
of moduli of K3 surfaces
which led to the present work,
and Yuji Odaka and Yoshiki Oshima
for exchange of ideas and results
on Gromov--Hausdorff compactifications
of moduli of K3 surfaces.
We also thank the anonymous referees
for pointing out mistakes in the older version
and suggesting several improvements,
and Remke Kloosterman
for a further correction and a remark
(see \pref{rm:local_torelli}).
K.~H.~was partially supported by Grants-in-Aid for Scientific Research
(17K14156).
K.~U.~was partially supported by Grants-in-Aid for Scientific Research
(24740043,
15KT0105,
16K13743,
16H03930).

\section{An elliptic surface}
 \label{sc:surfaces}

Let $h \in \bC[x,y]$ be the homogeneous polynomial of degree 24
defined by homogeneous polynomials
$f, g \in \bC[z,w]$ of degree 8 and 12 respectively
as in \pref{eq:def_h}.
Let further $\Xbar$ be the hypersurface
of degree 24
defined by
\begin{equation} \label{eq:def_Xbar}
 y^2=-x^3+h
\end{equation}
in the weighted projective space
$\bP(8,12,1,1) = \Proj \bC[x,y,z,w]$,
where the variables $x$, $y$, $z$, and $w$
are of degree $8$, $12$, $1$, and $1$ respectively.
The variety $\Xbar$ has a quotient singularity
of type
$
 \bC^2 \left/ \la \frac{1}{4}(1,1) \ra \right.
$
at $[x:y:z:w] = [-1:1:0:0]$
coming from the ambient space.
The exceptional divisor $\sigma_0$
of the minimal resolution
$X \to \Xbar$
is a $(-4)$-curve.
The $zw$-projection $\pibar \colon \Xbar \dashrightarrow \bP^1$
induces an elliptic fibration $\pi \colon X \to \bP^1$
such that $\sigma_0$ is a section.

Let $e$ be a smooth fiber of $\pi$,
which is the total transform of a hyperplane section of $\Xbar$.
The sublattice $U$ of $\Pic X$
generated by $e$ and $\sigma_0$
is the hyperbolic unimodular lattice of rank 2.

Note that
$
 K_X=\mathcal{O}_X(2e).
$
Any section $\sigma$ of $\pi$ has self-intersection $-4$
since
\begin{equation}
 \cN_{\sigma/X}
  = \cO_\sigma \lb K_\sigma - K_X|_\sigma \rb
  = \cO_\sigma(-2-2)
  = \cO_\sigma(-4).
\end{equation}
If $\sigma$ does not intersect $\sigma_0$,
then one has
$
 \sigma^2 = -4,
$
$
 \sigma \cdot \sigma_0 = 0,
$
and
$
 \sigma \cdot e = 1,
$
so that
\begin{align}
 \tau(\sigma) \coloneqq \sigma - \sigma_0 - 4 e
\end{align}
satisfies
$
 \tau(\sigma) \bot U
$
and
$
 \tau(\sigma)^2 = - 8.
$

The elliptic fibration $\pi$ has six sections
\begin{align} \label{eq:six_sections}
\begin{split}
 \sigma_1&: (x,y)=(f,g), \\
 \sigma_2&: (x,y)=(\zeta_3 f,g), \\
 \sigma_3&: (x,y)=(\zeta_3^2 f,g), \\
 \sigma_4&: (x,y)=(f,-g), \\
 \sigma_5&: (x,y)=(\zeta_3 f,-g), \\
 \sigma_6&: (x,y)=(\zeta_3^2 f,-g)
\end{split}
\end{align}
disjoint from $\sigma_0$,
where
$
 \zeta_k \coloneqq \exp \lb 2 \pi \sqrt{-1}/k \rb
$
for a positive integer $k$.
Let $M$ be the sublattice of $\Pic X$
generated by $\tau_i \coloneqq \tau(\sigma_i)$
for $i = 1, \ldots, 6$.

\begin{lemma} \label{lm:M}
The pair $(\tau_1, \tau_2)$ is an ordered basis of $M$
with the Gram matrix
$\sqmat{-8}{4}{4}{-8}$.
\end{lemma}

\begin{proof}
It follows from
\begin{align}
 \sigma_1 \cdot \sigma_4 = \deg g = 12
\end{align}
that
\begin{align}
 (\tau_1 + \tau_4)^2 = 0
\end{align}
and hence
\begin{align}
 \tau_1 + \tau_4 = 0
\end{align}
since $U^\bot \subset \Pic X$ is negative definite
by the Hodge index theorem.
Similarly,
it follows from
\begin{align} \label{eq:sigma12}
 \sigma_1 \cdot \sigma_2
  = \sigma_2 \cdot \sigma_3
  = \sigma_1 \cdot \sigma_3
  = \deg f = 8
\end{align}
that
\begin{align}
 (\tau_1 + \tau_2 + \tau_3)^2 = 0
\end{align}
and hence
\begin{align}
 \tau_1 + \tau_2 + \tau_3 = 0.
\end{align}
One also has $\tau_2 + \tau_5 = \tau_3 + \tau_6 = \tau_4 + \tau_5 + \tau_6 = 0$,
so that $\{ \tau_1, \tau_2 \}$ is a basis of $M$.
It also follows from \pref{eq:sigma12} that
\begin{align}
 \tau_1 \cdot \tau_2
  &= 4,
\end{align}
and \pref{lm:M} is proved.
\end{proof}

\section{The Picard number of $X$}
 \label{sc:rho}

Let
$
 \rho(X)
$
be the Picard number of $X$.
We prove the following in this section:

\begin{proposition} \label{pr:Pic}
For very general $f$ and $g$,
one has $\rho(X) = 4$.
\end{proposition}

\begin{proof}
Let $\cY$ be the family of elliptic surfaces over $\Spec \bC[a,b]$
obtained as the simultaneous minimal resolution
of the quotient singularity coming from the ambient space
of the family of hypersurfaces of $\bP(8,12,1,1)$
defined by \pref{eq:def_h}, \pref{eq:def_Xbar}, and
\begin{align}
 f = a w^8, \quad
 g = w \lb z^{11} + b w^{11} \rb
\end{align}
as in the beginning of \pref{sc:surfaces}.
One has
\begin{align}
 h
  = f^3 + g^2
  = w^2 \lb z^{11} - a' w^{11} \rb \lb z^{11} - b' w^{11} \rb,
\end{align}
where $a', b' \in \bC$ are defined by
\begin{equation}
 a' + b' = -2 b, \quad a' b' = a^3 + b^2.
\end{equation}
The discriminant of the elliptic fibration
is given by $h^2$,
which is a configuration of (not necessarily distinct) 24 double points.
The family $\cY$ is not isotrivial,
since the configuration depends on the parameter,
even after quotienting out the $\PGL(2, \bC)$-action.
A general member $Y$ of this family
has a singular fiber of Kodaira type I\!V at $[z:w]=[1:0]$,
consisting of three lines meeting at one point,
so that one has
\begin{align} \label{eq:rho_Y}
 \rho(Y) \geq 4 + 2 = 6.
\end{align}
Since the discriminant is of degree 48,
 the topological Euler number of $Y$ is 48,
 which implies that the second Betti number of $Y$ is given by
$
 48 - 2 = 46.
$
Since $Y$ is a weighted projective hypersurface,
the Griffiths--Dwork method
(see e.g. \cite{MR704986})
shows
\begin{align}
 H^{2,0}(Y) = \bigoplus_{i=0}^{2} \bC \Omega_i,
\end{align}
where
\begin{align}
 \Omega_i = \Res_Y \lb z^{i} w^{2-i}
  \frac{8 x d y \wedge d z \wedge d w - 12 y d x \wedge d z \wedge d w + \cdots}
  {x^3+y^2-h} \rb
\end{align}
for $i = 0, 1, 2$.
Hence the $\bZ/33\bZ$-action generated by
\begin{align}
 \alpha &\colon [x:y:z:w] \mapsto [\zeta_3 x:y:z:w], \\
 \beta &\colon [x:y:z:w] \mapsto [x:y:\zeta_{11} z:w]
\end{align}
satisfies
\begin{align} \label{eq:action_Omega}
 \alpha^* \Omega_i = \zeta_3 \Omega_i, \quad
 \beta^* \Omega_i = \zeta_{11}^{i+1} \Omega_i.
\end{align}
Let $V$ be the irreducible representation of $\bZ / 33 \bZ$
over $\bQ$ with eigenvalues $\lc \zeta_{33}^j \rc_{(j,33)=1}$.
One has
$
 \dim V = \phi(33) = 20,
$
where $\phi$ is Euler's totient function.
It follows from \pref{eq:action_Omega}
that $H^2(Y,\bQ)$ contains $V^{\oplus k}$ for some $k \geq 1$.
If $k=1$, then $H^{2,0}(Y)$ does not depend on the parameters $a$ and $b$,
which contradicts the non-isotriviality of the family $\cY$
and the local Torelli theorem for elliptic surfaces
 \cite{MR0499325,MR756851,MR721378,MR2043402}.
Hence one has $k=2$,
which implies
\begin{equation} \label{EQ_pic-Y}
 \rho(Y)
  \le 46 - \dim V^{\oplus 2}
  = 6,
\end{equation}
so that $\rho(Y) = 6$
for very general $Y$.

Now let $X$ be the elliptic surface
defined by very general $f$ and $g$,
so that any singular fiber is of type I\!I and $\rho(X)\leq 6$.
(For example, if we put $f=z^8$, $g=w^{12}$,
then we have $h=z^{24}+w^{24}$.
This implies that all roots of $h$ are pairwise distinct
for general $f$ and $g$.)
Assume $\rho(X)=6$ for a contradiction.
Then there is a deformation of $X$ to $Y$ such that
 $H^{2,0}(X) \subset V^{\oplus 2} \otimes \bC$
 under the induced identification
 $H^2(X;\bZ)=H^2(Y;\bZ)$,
 which implies $\Pic X = \Pic Y$.
Let $d$ be a $(-2)$-vector in $\Pic Y$
coming from the singular fiber of type I\!V.
The Riemann--Roch theorem shows
\begin{align}
 \chi(\cO_X(d))
  &\coloneqq h^0(\cO_X(d))-h^1(\cO_X(d))+h^2(\cO_X(d)) \\
  &= \frac{1}{2}d.(d-K_X)+\frac{1}{12}(K_X^2+c_2(X)) \\
  &= \frac{1}{2}d^2+\frac{1}{12} c_2(X) \\
  &= -1 + \frac{48}{12} \\
  &=3,
\end{align}
so that either $d$ or $K_X - d = - d + 2 e$ is effective.
Since $d$ is orthogonal to $e$
and every singular fiber of $X$ is irreducible,
the divisor $d$ must be a multiple of $e$,
which contradicts $d^2 = -2$ and $e^2 = 0$.

Next assume for a contradiction that $\rho(X)=5$.
Then $\Pic X \otimes \bQ$ is generated over $\bQ$ by
$U$, $M$ and $\delta$ with $\alpha^* \delta=\delta$.
By the theory of elliptic surfaces \cite{MR1081832}, the Mordell--Weil group of $X$ is
isomorphic to
\begin{equation}
\Pic(X)/(U + (\text{the lattice generated by irreducible components of singular fibers})).
\end{equation}
Hence the class $\delta$ corresponds to an $\alpha$-invariant section
different from $\sigma_0$.
Such a section is given by $x = 0$ and $y = \psi$
for a homogeneous polynomial $\psi$ in $z$ and $w$
of degree 12
satisfying $\psi^2 = h$.
The existence of a square root $\psi$ of $h$ contradicts the assumption
that $f$ and $g$ are very general,
and $\rho(X) = 4$ is proved.
\end{proof}

\begin{remark} \label{rm:local_torelli}
In the published version of this paper,
the authors have cited only \cite{MR0499325,MR756851,MR721378}
for the local Torelli theorem for elliptic surfaces.
Kloosterman pointed out that
the results in these papers are not strong enough
to deduce the local Torelli theorem for the family $\cY$,
but \cite[Theorems 1.1 and 3.3]{MR2043402} are.
He also pointed out that \pref{th:main} follows from the proof of
\cite[Proposition 2.1]{MR1971517}.
\end{remark}

\section{Proof of the main theorem}
 \label{sc:proof}

We prove \pref{th:main} in this section.
We first prove the uniqueness of the decomposition
for very general $f$ and $g$.
For any elements $\varphi$ and $\psi$
in $\bC[z,w]$
of degrees 8 and 12
satisfying
\begin{equation}
 h = \varphi^3 + \psi^2,
\end{equation}
the section $\sigma$ defined by $(x,y)=(\varphi,\psi)$
does not intersect $\sigma_0$,
so that $\tau(\sigma) \in U^\bot \subset \Pic X$.
For very general $f$ and $g$,
one has $\rho(X) = 4$,
and hence $\tau(\sigma) \in M \otimes \bQ$.
Recall that the Gram matrix
of the ordered basis $(\tau_1, \tau_2)$ of $M$
is $\sqmat{-8}{4}{4}{-8}$.
By a direct computation,
one can see that
there are no elements $\rho \in M \otimes \bQ$
such that $\rho.M\subset \bZ$ and $\rho^2=-8$
other than $\tau_1, \ldots, \tau_6$,
and the uniqueness of the decomposition follows from the fact that
the Mordell--Weil group of $X$ is naturally isomorphic to
$\Pic(X)/U\cong (M \otimes \bQ) \cap \Pic(X)$.

In order to prove the uniqueness of the decomposition
for general $f$ and $g$,
let $S$ and $T$ be the subspaces of $\bC[z,w]$
consisting of homogeneous polynomials of degrees $8$ and $12$
respectively, and
define a subscheme $Z$ of $(S \times T)^2$ by
\begin{equation}
 Z \coloneqq \lc ((f,g),(\varphi,\psi)) \in (S \times T)^2 \relmid
 f^3+g^2=\varphi^3+\psi^2 \rc.
\end{equation}
The uniqueness of the decomposition for very general $f$ and $g$ implies that
the first projection
$
 Z \to S \times T,
$
which is a morphism of schemes,
is generically six-to-one.
Hence the first projection is six-to-one
outside of a Zariski closed subset,
and \pref{th:main} is proved.

\section{Mirror symmetry}
 \label{sc:mirror}

It is conjectured by Strominger, Yau, and Zaslow
\cite{MR1429831}
that any Calabi--Yau manifold has a special Lagrangian torus fibration,
and the mirror manifold is obtained as the dual special Lagrangian torus fibration.
This picture has been refined in \cite{MR1863732,MR1882331}
to the conjecture that the Calabi--Yau metric
with the diameter normalized to one
converges in the Gromov--Hausdorff topology
to a Monge--Amp\`ere manifold with singularities
as one approaches a large complex structure limit.
Here, a \emph{Monge--Amp\`ere manifold with singularities}
is a manifold $B$
with a subset $B^\sing$ of Hausdorff codimension 2
such that $B \setminus B^\sing$
has a tropical affine structure
(i.e., an atlas whose transformation functions
are in $\GL_n(\bZ) \ltimes \bR^n$)
and a Monge--Amp\`ere metric
(i.e., a Riemannian metric
of Hessian form
$
 g_{ij} = \frac{\partial^2 K}{\partial x^i \partial x^j}
$
in the affine coordinate
satisfying
$
 \det \lb \lb g_{ij} \rb_{i,j} \rb = \text{constant}
$).
A Monge--Amp\`ere manifold comes
with a dual pair of affine structures,
and mirror symmetry should interchange them.

In the case of a K3 surface,
a special Lagrangian torus fibration can be turned
into an elliptic fibration by a hyperK\"ahler rotation.
It is shown in \cite{MR1863732,MR3514555}
that a sequence of Calabi--Yau metrics
on an elliptic K3 surface
with a fixed diameter
converges
in the Gromov--Hausdorff topology
to a sphere with a Monge--Amp\`ere structure with singularities
as the volume of the fiber goes to zero.
The limit sphere $B$ can naturally be identified
with the base $\bP^1$ of the elliptic K3 surface,
and the discriminant $B^\sing$
of the Monge--Amp\`ere structure
can be identified with the discriminant
of the elliptic K3 surface.
Under this identification,
the Monge--Amp\`ere metric $g$ on $B \setminus B^\sing$
is written as
\begin{align} \label{eq:MAS2}
 g = \Im \lb \taubar_1 \tau_2 \rb d z \otimes d \zbar,
\end{align}
where $z$ is the holomorphic local coordinate on the base $\bP^1$
and $(\tau_1, \tau_2)$ are the periods,
along a symplectic basis,
of the relative holomorphic one-form $\lambda$
on the elliptic fibration
dual to $d z$ with respect to the holomorphic volume form of the K3 surface.
It follows that the complex structure of the base $\bP^1$
minus the discriminant
can be reconstructed from the limit metric,
up to the choice of an orientation.
Note that the metric \pref{eq:MAS2} depends only on the Jacobian fibration,
so that one can assume that the elliptic K3 surface has a section.
It follows from \pref{th:main} that a general elliptic K3 surface with a section
can be reconstructed from the limit metric
up to complex conjugation.

\bibliographystyle{amsalpha}
\bibliography{bibs}

\end{document}